\documentclass[12pt,twoside]{article}
\usepackage{amsmath,amsfonts,amsfonts,amssymb}
\usepackage{amsbsy,theorem}
\usepackage{mathrsfs}
\usepackage{amscd}
\usepackage{pstricks}
\usepackage{color}
\usepackage{graphicx}

\theoremstyle{break}
\newtheorem{lemma}{Lemma}[section]

\newtheorem{proposition}[lemma]{Proposition}
\newtheorem{theorem}[lemma]{Theorem}

\theorembodyfont{\upshape}

\newcommand \QQ {{\mathbb Q}}
\newcommand \PP {{\mathbb P}}

\newcommand \FF {{\mathbb F}}

\newcommand \AAA {{\mathbb A}}
\newcommand \RR {{\mathbb R}}
\newcommand \ZZ {{\mathbb Z}}
\newcommand \LL {{\mathbb L}}
\newcommand \cI {{\mathcal I}}
\newcommand \cM {{\mathcal M}}

\newcommand \cV {{\mathcal V}}

\newcommand{\scs}{\scriptstyle}

\newcommand{\s}[1]{\ensuremath{\scs{#1}}}
\newcommand{\so}{\ensuremath{\scs{0}}}
\newcommand{\sk}{$\scs{\!\!-1}$}
\newcommand{\sj}{$\scs{1}$}

\begin{document}
\begin{center}
    \LARGE
    {\bf
    On one example and one\\
    counterexample in counting rational\\
    points on graph hypersurfaces.
    \vspace{2ex}\\
    }
    \large\sc{Dzmitry Doryn}
  \end{center}
\medskip
\begin{abstract}
\noindent In this paper we present a concrete counterexample to the
conjecture of Kontsevich about the polynomial countability of graph
hypersurfaces. In contrast to this, we show that the "wheel with
spokes" graphs $WS_n$ are polynomially countable.
\end{abstract}

\section*{Introduction} \addcontentsline{toc}{chapter}{Introduction}

Let $\Gamma$ be a connected graph with a set of edges $E$ and
vertexes $V$. We can define the graph polynomial of $G$ by
\begin{equation}
    \Psi_{\Gamma}:=\sum_T \prod_{e\notin T} A_e \in \ZZ[A_e|e\in E],
\end{equation}
where the sum goes over all spaning trees and $A_e$'s are variables.
The vanishing locus of such a polynomial in the affine space
$\AAA^{|E|}_\ZZ$ (or in $\PP^{|E|}_\ZZ$) defines the graph
hypersurface $X_{\Gamma}$. It is convenient for us to use this
affine notation since the big part of computations was done with PC.

We are interested in the number of $\FF_q$-rational points of graph
hypersurface. Consider a function $F_{\Gamma}:q\mapsto
\#X_{\Gamma}(\FF_q)$ defined on the set of prime powers. We known
that all know periods of Feynman graphs are elements of a
$\QQ$-subalgebra of $\RR$ generated by multiple zeta values (see
\cite{BrK} and \cite{Sch}). There was a hope that the arithmetic of
graph hypersurfaces is very simple. In '97 M.Kontsevich made the
following conjecture on the number of points of graph hypersurfaces:
\emph{For any graph} $\Gamma$, $F_{\Gamma}\in\ZZ[q]$. The conjecture
was wrong. That was proved in \cite{BB} in '00.  The proof is very
technical and gives no concrete example of a graph for which the
function $F_{X_{\Gamma}}$ is not a polynomial.

In this paper we give a concrete counterexample with a strong proof.
There are parallel computations done by Oliver Schnetz
independently, where he proposed 6 counterexamples in $\phi^4$
theory with 14 edges. Focusing on finding of all counterexamples in
$\phi^4$ theory, he computes $F_{\Gamma}(q)$ for $q$ up to 7 and
finds the graphs, for which the coefficients after (partial)
interpolation in these points become very large; he speculates on
the good behaviour of values $F_{\Gamma}(q)$ at prime powers $q=2^k$
or away from this set (5 examples), and with respect to the residue
modulo 3 (one other example).

Actually, if you are sure about a graph being a counterexample, you
can use cluster computer system to proof this. My aim is to show how
it was done for our example with a home computer using the
stratification of graph hypersurface studied in \cite{BEK} and
\cite{DD}. My graph was chosen by some thoughts of symmetry and
after some tries, and one of the Oliver's graphs is isomorphic to my
counterexample.

You can find the implementation of the algorithm on\\
$\mathstrut\quad\quad\quad\underline{http://doryn.org/progs/conterexample.html}$.

The consequence of the result of (\cite{BB}) is that the
$F_{\Gamma}$ is not a polynomial for almost all graphs. On the other
hand, Spencer Bloch proves in \cite{Bl} that the (finite) sum of all
$F_{\Gamma}$ for connected graphs with $n$ edges (counted with some
multiplicities) is a polynomial for any $n\geq 3$. Can the big
enough graph be polynomially countable? The example is the cycle
$O_n$ of length $n$, since $F_{O_n}(q)=q^{n-1}$. We avoid such
trivial cases. Recall that primitively log divergent graphs are the
graphs $\Gamma$ such that $|E(\Gamma)|=2h_1(\Gamma)$ and for all
subgraphs $\Gamma'\subset\Gamma$ the inequality
$E(\Gamma')>2h_1(\Gamma')$ holds. We restrict our attention to the
primitively log divergent graphs since for such graphs the periods
are defined. So, the natural question is whether we have a
primitively log divergent graph $\Gamma_n$ with the Betti number
$h_1(\Gamma)=n$ such that $F_{X}$ is a polynomial for each $n\geq
3$. The answer is yes. The function $F_{\Gamma}$ is a polynomial for
each graph from the $WS_n$ series. In the last part of the paper we
prove this statement by computing the class of the graph
hypersurface of $WS_n$ in the Grothendick ring of varieties.

\section{Counterexample}
We start with presenting our graph $\Gamma=XStrip$. It looks like a
strip of 3 squares and with $X's$ inside both pairs of consequent
squares (see drawing). For a graph polynomial we use the
presentation as a determinant of a matrix. We choose some
orientation of edges and the direction of loop tracing, and we build
an $h_1(\Gamma)\times N$ -- table $Tab(\Gamma)$ with
$N=|E(\Gamma)|$. The $Tab(\Gamma)_{ij}=1$ if the edge $e_j$ in the
$i$'s loop is in the tracing direction of the loop and
$Tab(\Gamma)_{ij}=-1$ if this edge is in the opposite direction;
otherwise $Tab(\Gamma)_{i j}=0$. Then the desired matrix is
$\cM_{\Gamma}(T):=\sum_{k=1}^N T_k \cM^k$ in some variables
$T_0,\ldots,T_N$, where $\cM^k_{i j}=Tab(\Gamma)_{i k}\cdot
Tab(\Gamma)_{j k}$. For more details see (\cite{DD}, Example 1.2.5).
\bigskip\\
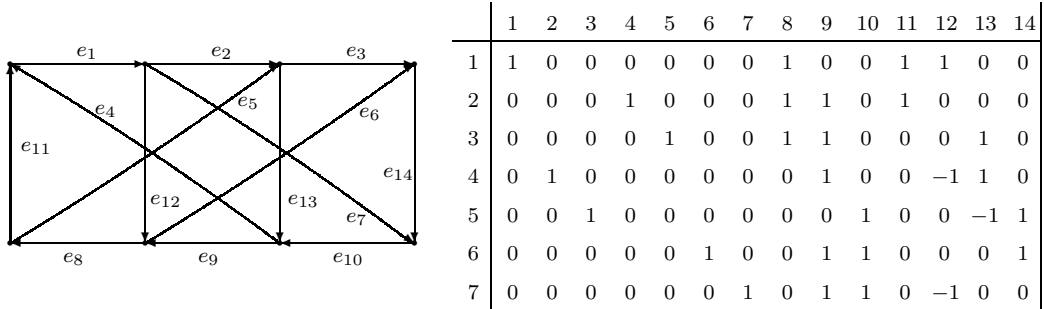
\begin{picture}(0,0)(0,40)
\put(153,10){\vector(-1,0){51}} \put(0,10){\vector(0,1){68}}
\put(102,10){\vector(-1,0){51}} \put(51,10){\vector(-1,0){51}}
\put(0,78){\vector(1,0){51}} 
\put(51,78){\vector(1,0){51}} \put(102,78){\vector(1,0){51}}
\put(51,78){\vector(0,-1){68}}
\put(153,78){\vector(0,-1){68}} \put(102,78){\vector(0,-1){68}}
\put(0,10){\circle*{2}} \put(0,78){\circle*{2}}
\put(51,10){\circle*{2}} \put(51,78){\circle*{2}}
\put(102,10){\circle*{2}} \put(102,78){\circle*{2}}
\put(153,10){\circle*{2}} \put(153,78){\circle*{2}}

\qbezier(0,78)(51,48)(102,10) \qbezier(0,10)(51,41)(102,78)
\qbezier(51,78)(102,48)(152,10) \qbezier(51,10)(102,41)(153,78)

\put(0,78){\vector(-3,2){0}} \put(102,78){\vector(3,2){0}}
\put(153,10){\vector(3,-2){0}} \put(153,78){\vector(3,2){0}}

\put(4,45){$\s{e_{11}}$} 
\put(25,81){$\s{e_1}$} \put(76,81){$\s{e_2}$}
\put(127,81){$\s{e_3}$} \put(32,59){$\s{e_4}$} \put(20,3){$\s{e_8}$}
\put(71,3){$\s{e_9}$} \put(122,3){$\s{e_{10}}$}

\put(141,35){$\s{e_{14}}$} \put(105,25){$\s{e_{13}}$}
\put(53,25){$\s{e_{12}}$} \put(86,62){$\s{e_5}$}
\put(132,58){$\s{e_6}$} \put(127,18){$\s{e_7}$}

\end{picture}
\mathstrut\qquad\qquad\qquad\qquad\qquad\qquad\qquad
\begin{tabular}{p{0.1cm}|p{0.1cm}p{0.1cm}p{0.1cm}p{0.1cm}p{0.1cm}p{0.1cm}p{0.1cm}
p{0.1cm}p{0.1cm}p{0.1cm}p{0.1cm}p{0.1cm}p{0.1cm}p{0.1cm}|}
      &\s{1} &\s{2} &\s{3} &\s{4} &\s{5} &\s{6} &\s{7} &\s{8} &\s{9}
      &\s{\!10}&\s{\!11}&\s{\!12}&\s{\!13}&\s{\!14}\\\hline
 \s{1}&\sj  &\so  &\so  &\so  &\so  &\so  &\so  &\sj  &\so  &\so  &\sj  &\sj  &\so &\so\\
 \s{2}&\so  &\so  &\so  &\sj  &\so  &\so  &\so  &\sj  &\sj  &\so  &\sj &\so  &\so &\so\\
 \s{3}&\so  &\so  &\so  &\so  &\sj  &\so  &\so  &\sj  &\sj  &\so  &\so &\so  &\sj &\so\\
 \s{4}&\so  &\sj  &\so  &\so  &\so  &\so  &\so  &\so  &\sj  &\so  &\so &\sk  &\sj &\so\\
 \s{5}&\so  &\so  &\sj  &\so  &\so  &\so  &\so  &\so  &\so  &\sj  &\so &\so  &\sk &\sj\\
 \s{6}&\so  &\so  &\so  &\so  &\so  &\sj  &\so  &\so  &\sj  &\sj  &\so &\so  &\so &\sj\\
 \s{7}&\so  &\so  &\so  &\so  &\so  &\so  &\sj  &\so  &\sj  &\sj  &\so &\sk  &\so &\so\\
\end{tabular}
To have a better chances in fighting with polynomials, we change the
variables to have as much independent entries of the matrix as
possible. Then we come to (denoting by the same letter) matrix $\cM$
in variables $A_0,\ldots,A_6$, $B_0,\ldots,B_6$.
\smallskip\\
\begin{equation}
\cM_{\Gamma}(A,B)=
 \left(
  \begin{array}{cccccccc}
     \s{B_0} & \s{A_0} & \s{A_2}    & \s{A_1} &\so  &\so & \s{A_1} \\
     \s{A_0} & \s{B_1} & \s{A_2+A_3}   &\s{A_3}   & \so  &\s{A_3} &\s{A_3} \\
     \s{A_2} & \s{A_2+A_3} &  \s{B_2}  &\s{A_3-A_4}   & \s{A_4} & \s{A_3} &\s{A_3} \\
     \s{A_1} & \s{A_3} & \s{A_3-A_4} &\s{B_3} & \s{A_4}  &\s{A_3} &\s{A_3-A_1}\\
     \so     & \so    & \s{A_4}  &\s{A_4}  &\s{B_4} & \s{A_5} & \s{A_6}\\
     \so     &  \s{A_3} & \s{A_3} & \s{A_3} & \s{A_5}  & \s{B_5} & \s{A_3+A_6}\\
     \s{A_1} & \s{A_3}  & \s{A_3}  & \s{A_3-A_1} & \s{A_6} & \s{A_3+A_6} & \s{B_6}\\
  \end{array}
\right) \label{mat5}
\end{equation}
In this section we use the affine notion of graph hypersurface. So,
$X:=X_{\Gamma}\subset\AAA^{14}(A,B)$ defined by $\det(\cM)=0$ in the
affine space with coordinates all of $A$'s and $B$'s, where
$\cM:=\cM_{\Gamma}(A,B)$. We write $X=\cV(\det(\cM))$ in this
situation. More generally, we denote by $\cV(\cI)$ or
$\cV(f_1,\ldots,f_n)$ the variety in $A^N(T_1,\ldots,T_N)$ defined
by the vanishing locus of the ideal generated by polynomials
$f_1,\ldots,f_n\in\ZZ[T_1,\ldots,T_N]$. Sometimes we write
$\cV(\cI)^{(N)}$ indicating the dimension $N$ of the ambient affine
space.

Consider the function $F_{\Gamma}:q\mapsto\#X_{\Gamma}(\FF_q)$. The
core of this article is the following
\begin{theorem}
    If $\Gamma=XStrip$, then $F_{\Gamma}$ is not a polynomial.
\end{theorem}

Assume that $F_{\Gamma}(q)$ is a polynomial. The proof is based on
the computer program, but there are several steps of optimization
needed to get the answer in a reasonable time. We use a shape of the
matrix and make a stratification of graph hypersurface.\bigskip\\
\textbf{\emph{Step 1.}} First we explain the simple projection
techniques used (\cite{BEK}) and (\cite{DD}). The polynomial
$I_7=M:=\det(\cM)$ is linear in $B_0$: $I_7=B_0I_{6}-G_{6}$ with
$G_{6}:=-I_7|_{B_0=0}$. If $I_{6}=0$, then the equation $I_7=0$
implies $G_{6}=0$ and we forget the variable $B_0$. So, the good
idea is to consider the image of $X\cap I_{6}$ under the "forgetting
$B_0$" projection from $\AAA^{14}$ to the affine space of one less
dimension $\AAA^{13}(A,B;$ no $B_0)$. In the other case
--- when $I_{n-1}\neq 0$ --- we can express $B_0$ from the equation
$I_n=0$ and also project to $A^{13}$ getting an isomorphism
$X\backslash X\cap\cV(I_6)\cong \AAA^{13}\backslash\cV(I_6)^{(13)}$.

Consider the the class $[X]$ of the graph hypersurface in the
Grothendick ring $K_0(Var_{K})$ of varieties over a field $K$ of
characteristic 0. By definition,
\begin{equation}
    [X]=[\cV(I_7)]=[X\cap\cV(I_6)]+[X\backslash X\cap\cV(I_6)].
    \label{b2}
\end{equation}
Using the explained projections, one gets
\begin{equation}
    [X\cap\cV(I_6)]=[\cV(I_6,I_7)^{(14)}]= [ \LL ] [\cV(I_6,G_6)]
    \label{b4}
\end{equation}
with $\cV(I_6,G_6)$ living in $\AAA^{13}(A$,$B$, no $B_0$) and
$\LL=[\AAA^1]$. Also
\begin{equation}
    [X\backslash
    X\cap\cV(I_6)]=[\AAA^{13}\backslash\cV(I_6)^{(13)}]=[\LL][\AAA^{12}\backslash\cV(I_6)]=
    \LL^{13}-\LL[\cV(I_6)]
    \label{b5}
\end{equation}
with $\cV(I_6)\subset\AAA^{12}(A,B;$ no $B_0$,$A_0)$ since $I_6$ is
independent of $A_0$. The polynomials $I_5$ and $G_5$ are
independent of $B_1$, while $I_5$ is also independent of $A_2$.
Thus, repeating the procedure for $\cV(I_6)$, we obtain
\begin{multline}
[\cV(I_6)]=[\cV(I_6,I_5)]+[\cV(I_6)\backslash\cV(I_6,I_5)]=
[\cV(I_5,G_5)^{(12)}]+\\
[\AAA^{11}\backslash\cV(I_5)^{(11)}]=\LL[\cV(I_5,G_5)]+\LL[\AAA^{10}\backslash\cV(I_5)],
\end{multline}
where $\cV(I_5,G_5)\subset\AAA^{11}$($A$,$B$, no $B_0$,$B_1$,$A_0$)
and $\AAA^{10}\backslash\cV(I_5)\subset \AAA^{10}$($A$,$B$ no
$B_0$,$B_1$,$A_0$,$A_2$).

 Now we consider $G_6$ as a quadratic polynomial of variables
$A_0$, $A_1$ and $A_2$ (sitting in the first row and column). The
coefficient of $A_0$ is then $I_5$. Now we use two tricks from the
paper \cite{DD}. First, by Corollary 1.5, the product polynomial
$G_6I_5$ is a square of a linear polynomial in $A_0$, $A_1$, $A_2$
on the locus where $I_6=0$ and $I_5\neq 0$. Using this one can
express $A_0$ from that linear polynomial on
$\cV(I_6)\backslash\cV(I_6,I_5)$ and get rid of $G_6$. Second is
Theorem 1.6: if $I_5=0$, then $G_6$ has not just the zero
coefficient of $A_0^2$, but is independent of $A_0$ on $\cV(I_6)$ at
all. It follows that
\begin{multline}
    [\cV(I_6,G_6)]=[\cV(I_5,I_6,G_6)]+[\cV(I_6,G_6)\backslash\cV(I_5,I_6,G_6)]=\\
    \LL[\cV(I_5,I_6,\tilde{G}_6)]+[\AAA^{11}\backslash\cV(I_5)^{(11)}]=
    \LL[\cV(I_5,I_6,\tilde{G}_6)]+\LL[\AAA^{10}\backslash\cV(I_5)],
\end{multline}
where $\tilde{G}_6:=G_6|_{A_0=0}$. Collecting everything together,
we obtain
\begin{multline}
    [X]=\LL[\cV(I_6,G_6)]+ \LL^{13}-\LL[\cV(I_6)] =
    \LL\left(\LL[\cV(I_5,I_6,\tilde{G}_6)]+
    \LL[\AAA^{10}\backslash\cV(I_5)]\right)\\
    +\LL^{13}-\LL\left(\LL[\cV(I_5,G_5)]+\LL[\AAA^{10}\backslash\cV(I_5)]\right)=\\
    \LL^{13}-\LL^2[\cV(I_5,G_5)] +\LL^2[\cV(I_5,I_6,\tilde{G}_6)].
    \label{b7}
\end{multline}
\textbf{\emph{Step 2.}} The formula above is interesting as itself
(holds in general for all primitively divergent graph) and we will
use similar technique to optimize the algorithm of computation
further, but the one direct consequence is the following.
\begin{proposition}\label{divis1}
    The polynomial $F_\Gamma(q)$ is divisible by $q^2$.
\end{proposition}
This is implied by the fact that the functor of counting rational
points factors through Grothendick ring of varieties.

Let us look closely at $F_\Gamma$ assumed being a polynomial. Since
$X$ is a hypersurface in $\AAA^{14}$, the degree of $F_X(q)$ is at
most 13. Recall that the hypersurface associated to a primitively
log divergent graph is always irreducible. This can be proved easily
by induction. As a consequence, the leading term of $F_X(q)$ is
$q^{13}$. By Proposition $\ref{divis1}$, we can rewrite $F_X(q)$
like
\begin{equation}\label{b9}
    F_X(q)=q^{13}+q^2\tilde{F}(q),
\end{equation}
where $\tilde{F}$ is polynomial of degree at most 10. Such
polynomial can be uniquely defined by its 11 values. So we need to
compute $\#X(F_q)$ for at least 12 prime powers. Of course, we take
the first prime powers starting with 2 and up to 19.

Our algorithm computes $\tilde{F}(19)$ in three days on the home
computer, but for this reason the formula (\ref{b7}) is not enough,
we need to stratify further using the shape of the matrix. It
becomes complicated, we do this in steps.\bigskip\\
\textbf{\emph{Step 3.}} We return to the formula \ref{b7}. The
polynomial $\tilde{G}_6$ is of degree 1 as a polynomial of $B_1$.
Write
\begin{equation}
    \tilde{G}_6=\tilde{G}^1_6B_1 + \tilde{G}^2_6.
\end{equation}
For
$\cV(I_5,I_6,\tilde{G}_6)=\cV(I_5,G_5,\tilde{G}_6)\subset\AAA^{12}$(A,B
no $A_0$,$B_0$), we separate into two cases according to whether
$\tilde{G}^1_6$ equals zero or not, and we get
\begin{multline}
    \quad[\cV(I_5,G_5,\tilde{G}_6)]=[\cV(I_5,G_5,\tilde{G}_6,\tilde{G}^1_6)]+\\
    [\cV(I_5,G_5,\tilde{G}_6)\backslash\cV(I_5,G_5,\tilde{G}_6,\tilde{G}^1_6)].\quad
\end{multline}
On the first variety on the right we forget $B_1$, while on the last
open scheme we can express $B_1$ from the equation $\tilde{G}_6=0$,
projecting down to $\AAA^{11}$. So, we obtain
\begin{multline}
    [\cV(I_5,G_5,\tilde{G}_6)]=\LL[\cV(I_5,G_5,\tilde{G}^1_6,\tilde{G}^2_6)]+
    [\cV(I_5,G_5)\backslash\cV(I_5,G_5,\tilde{G}^1_6)].
    \label{b11}
\end{multline}
By (\ref{b2}) and (\ref{b11}), one gets
\begin{multline}
    [X]=\LL^{13}-\LL^2[\cV(I_5,G_5)]+\LL^2[\cV(I_5,I_6,\tilde{G}_6)]
    =\LL^{13}-\LL^2[\cV(I_5,G_5)]\\
    +\LL^3[\cV(I_5,G_5,\tilde{G}^1_6,\tilde{G}^2_6)]
    +\LL^2([\cV(I_5,G_5)]
    -[\cV(I_5,G_5,\tilde{G}^1_6)])\\=\LL^{13}-\LL^2[\cV(I_5,G_5,\tilde{G}^1_6)]
    +\LL^3[\cV(I_5,G_5,\tilde{G}^1_6,\tilde{G}^2_6)]
    \label{b12}
\end{multline}
\textbf{\emph{Step 4.}} This step is more closely to the
implementation of the algorithm. The situation is the following. To
compute the number of rational points $\#X(\FF_q)$ we need to count
the number of solutions
$(a_0,\ldots,a_7,b_1,\ldots,b_7)\in\AAA^{14}$ of the equation
$I_7=0$. The brute force strategy is to put each of $q^{14}$
14-tuples into the equation and check if it is a solution. Then the
complexity is $O(14)$ times the complexity of one such check. If we
apply the Gauss algorithm for this, the total complexity will be
$O(14)O(3)=O(17)$. This is impossible to compute $\#X(\FF_q)$ with
this strategy for $q=19$ in 1 year with a home PC.

Formula (\ref{b12}) helps to restrict the complexity to
$O(11)O(3)=O(14)$ since we deal with 11-tuples $a_1,\ldots,b_7$ (no
$b_0$, $a_0$ or $b_1$) approximately. The complexity depends on the
branch of the algorithm where some polynomials vanish or not. The
last useful trick is the following. The polynomial $I_5=B_2I_4-G_4$
is zero on both varieties $Y=\cV(I_5,G_5,\tilde{G}^1_6)$ and
$Z=\cV(I_5,G_5,\tilde{G}^1_6,\tilde{G}^2_6)$ appearing in the
formula. We can decrease the number of tuples further. For each
10-tuple $a_1,\ldots,b_7$ (no $b_0$, $a_0$, $b_1$ or $b_2$) such
that $I_4\neq 0$ we have unique $b_2$ such that $I_5$ is zero. In
the opposite case, when $I_4=0$, we have no vanishing of $I_5$ if
$G_4\neq 0$, and the vanishing for all $b_2$ if $G_4=0$. Then on the
branch where $I_4\neq 0$ and $I_5=0$ we can compute $a_2$ in a way
to have $G_5=0$, so we do not need to run over all values of $a_2$.
Such tricks restrict the running time. For more explanation see the
program.

\section{Results}
The program gives us 13 points $F(2)$, $F(3)$,$\ldots$, $F(23)$. The
biggest one, $F(23)$ was computed in 2 weeks on my office computer
in Universit$\ddot{\textmd{a}}$t Duisburg-Essen on the Christmas
vacation 2008-2009. Then we use Lagrange interpolation formula for
first 11 of them to interpolate $\tilde{F}$ in (\ref{b9}). Then we
take $F(19)$ and check whether the point $(19,\tilde{F}(19))$ is
lying on that graph or not. It is not! After multiplying with
determinants the problem becomes integral and can be done even by
hand.

The interesting thing is that for the 10 odd prime powers q =3, 5,
7, 9, 11, 13, 17, 19, 23 the values $F(k)$ are exactly the values of
the following polynomial
\begin{multline}
F_1(q)=q^{13} +q^{11} +23q^{10} -78q^9 +90q^8 -35q^7 +(q-2)q^6\\
-34q^5 +66q^4 -32q^3 +(q-1)q^2
\end{multline}
But the values in the even prime povers 2, 4, 8, 16 are the values
of the polynomial $F_1-(q-1)q^2$. We believe that for $q$ prime to 2
the number of rational points $F_{X_{\Gamma}}(q)$ is $F_0(q)$, and
that for even prime powers
$F_{X_{\Gamma}}(2^k)=F_0(2^k)-(2^k-1)2^{2k}$. This will be not very
surprising since the first counterexample must be "not very bad",
and the simple picture for this is that we have a polynomial in all
but one prime. At this prime and all it's powers we have another
polynomial and the difference must be divisible by $q^2$ by Lemma
\ref{divis1}. The results are compatible with that of Schnetz,
\cite{Sch}, our graph is isomorphic to the graph on \emph{Fig 1.a}
without vertex 1. If one goes from the affine case to the projective
complement of graph hypersurface, he gets the polynomial
(2.30),\cite{Sch}.

\section{Example}

Here we study the well-known series $WS_n$. These are the simplest
examples of primitive log divergent graphs. The graphs $WS_n$,
$n\geq 3$ looks like $n$ points on the circle and one point inside,
the first $n$ edges connect the center to the points on the circle
and the other $n$ edges are the arcs that the circle is divided on.
To make the formulas a bit more readable later, we consider
$WS_{n+1}$.

In this section we work in projective setting and the graph
hypersurface $X$ lives in $\PP^{n+1}$. Now the$\cV(\cI)$ or
$\cV(f_1,\ldots,f_n)$ means the vanishing locus of the ideal
generated by homogenious polynomials
$f_1,\ldots,f_n\in\ZZ[T_1:\ldots:T_N]$ in projective space. If
$Y_p=\cV(f_1,\ldots,f_n)\subset \PP^{N-1}$ in this setting and
$Y_a=\cV(f_1,\ldots,f_n)\subset \AAA^{N}$ in the setting of Section
1, then the natural projection $\AAA^{N}\backslash
\{0\}\longrightarrow \PP^{N-1}$ induces an equality
$[Y_a]-1=(\LL-1)[Y_p]$. Thus $\#Y_a(\FF_q)=1+q\cdot\#Y_p$, and for
proving that $F_{\Gamma}$ is a polynomial it does not matter,
whether we work in the affine setting or in the projective one.

The graph polynomial of $WS_{n+1}$ is the determinant of the
following matrix.
\begin{equation}
 \cM_{n+1}=\cM_{WS_{n+1}}=
\left(
  \begin{array}{cccccccc}
    \s{B_0}& \s{A_0}& \so & \vdots & \so & \so &\s{\!A_{n}\!} \\
    \s{A_0}& \s{B_1}& \s{A_1}&\vdots& \so & \so  & \so \\
    \so    & \s{A_1}   & \s{B_2} &\vdots& \so &  \so & \so \\
    \ldots &\ldots &\ldots &\ddots &\ldots &\ldots &\ldots \\
    \so    & \so    & \so &\vdots& \s{\!B_{n-2}\!} &\s{\!A_{n-2}\!}  &\so \\
    \so    & \so    & \so &\vdots&\s{\!A_{n-2}\!}&\s{\!B_{n-1}\!} &\s{\!A_{n-1}\!}  \\
    \s{\!A_{n}\!} & \so&\so&\vdots&\so  &\s{\!A_{n-1}\!}&\s{\!B_{n}\!} \\
  \end{array}
\right).
\end{equation}

We again work the class of $X=\det{\cM_{n+1}}$ in the Grothendick
ring. Using the same technique as in the first section, we stratify
\begin{equation}
    [X]=\LL^2[\PP^{2n-2}\backslash\cV(I_{n})]+1+\LL[\cV(I_{n},G_{n})].
\end{equation}
Here $I_{n}$ is independent of $B_0,A_0$ and $A_{n}$; and
$G_{n}=-I_{n+1}|_{B_0=0}$. Since $I_{n+1}(0;n)=A_0A_1\ldots
A_{n-1}$, we can write
\begin{equation}
    G_{n}=A_0^2I_{n-1}-2A_0A_{n}\cdot A_0A_1\ldots
    A_{n-1}+A_{n}^2I_{n-1}^1
\end{equation}
Knowing $G_{n}$, we go on
\begin{multline}
    [X]=\LL^2([\PP^{2n-2}]-1)-\LL^3[\cV(I_{n-1},G_{n-1})]-\LL^3[\PP^{2n-4}\backslash\cV(I_{n-1})]+\\
    1+\LL+\LL^2[\cV(I_{n-1},I_{n},\tilde{G}_{n})]+\LL^3[\PP^{2n-4}\backslash\cV(I_{n-1})]=
    1+\LL+\\
    \LL^2([\PP^{2n-2}]-1)-\LL^3[\cV(I_{n-1},G_{n-1})]+\LL^2[\cV(I_{n-1},I_{n},\tilde{G}_{n})]=\\
    1+\LL+\LL^2([\PP^{2n-2}]-1)-\LL^3[Y']+\LL^2[Z'],
    \label{b16}
\end{multline}
where $G_{n-1}=-I_{n}|_{B_1=0}=A_1^2I_{n-2}$ and
$\tilde{G}_{n}=-I_{n+1}|_{B_0=0=A_0}=A_{n}^2I_{n-1}^1$.

First consider $Y'=\cV(I_{n-1},A_1I_{n-2})$. For each $i$, $1\leq i
\leq n-1$, we define
$Y'_i:=\cV(I_i,A_{n-i-1}I_{i-1})\subset\PP^{2i-1}$ in the projective
space with coordinates all $A's$ and $B's$ appearing in the the
submatrix of $I_i$ and $A_{n-i-1}$. Similarly $Y_i:=\cV(I_{i})$ is
defined in $\PP^{2i-2}$ with the same coordinates but no
$A_{n-i-1}$. In this notation $Y'=Y'_{n-1}$. Set $y_i:=[Y_i]$ and
$y'_i=[Y'_i]$ in $K_0(Var_K)$.

\begin{lemma}\label{recur1}
For $y'_i$ and $y_i$ as above, we have $y'_i,y_i\in \ZZ[\LL]$.
\end{lemma}
\begin{proof} We prove the statement by induction on $i$. For $i=1$,
$Y_1=\cV(A_{n-2},B_{n-1})\subset\PP^1$ and
$Y'_1=B_{n-1}\subset\PP^0$, so $y'_1=y_1=0\in\ZZ[\LL]$. Assume now
that for $i=k<n-2$ the statement is true. Then for $i=k+1$:
\begin{multline}
    [Y'_{k+1}]=[\cV(I_{k+1},A_{n-k-2}I_{k})]= [\cV(I_{k+1},A_{n-k-2})]+[\cV(I_{k+1},I_k)^{(2k)}]-\\
    [\cV(I_{k+1},A_{n-k-2},I_k)]=[Y_{k+1}]+1+\LL[\cV(I_{k+1},I_k)^{(2k-1)}]-[\cV(I_{k+1},A_{n-k-2},I_k)]=\\
    [Y_{k+1}]+ 1+(\LL-1)[\cV(B_{n-k-1}I_{k}-A_{n-k-1}^2I_{k-1},I_k)]= \\
    [Y_{k+1}]+ 1+(\LL-1)(1+\LL[\cV(I_k,A_{n-k-1}I_{k-1})])
    =[Y_{k+1}]+\LL(\LL-1)[Y'_k]+\LL
\end{multline}
\begin{multline}
    [Y_{k+1}]=[\cV(B_{n-k-1}I_k-A_{n-k-1}^2I_{k-1})]=(1+\LL[\cV(I_k,A_{n-k-1}I_{k-1})])+\\
    \LL[\PP^{2k-2}\backslash \cV(I_k)]=
    1+\LL[Y'_k]+\LL[\PP^{2k-2}]-\LL[Y_k];
\end{multline}
We get recurrence formulas
\begin{equation}
\begin{aligned}
y_{k+1}=1+\LL y'_k+\LL[\PP^{2k-2}]-\LL y_k,\\
y'_{k+1}=\LL+y_{k+1}+\LL(\LL-1)y'_k
\end{aligned}
\end{equation}
with $y'_1=y_1=0$. If one likes, the substitution $v_k:=[Y_k]$ and
$u_k=[Y'_k]-[Y_k]$ brings us to a one recurrence formula for $u_k$,
which is easy to compute
\begin{equation}
\begin{aligned}
    v_{k+1}&=1+\LL[\PP^{2k-1}]+\LL u_k,\\
    u_{k+1}&=\LL+\LL(\LL-1)(u_k+1+\LL u_{k-1}+\LL[\PP^{2k-3}]).
\end{aligned}
\end{equation}
Now elements $y_i$ and $y'_i$ are now evidently polynomials of
$\LL$.
The first terms are listed in the table below. \medskip\\
\mathstrut\quad\begin{tabular}{c|c|c|c|c|}
     i &  \;1\; & 2 & 3& 4 \\ \hline
    $y_i$&\;0\; & $\LL+1$ & $\LL^3+2\LL^2+\LL+1$ & $\LL^5+3\LL^4+\LL^2+\LL+1$ \\
    $y'_i$ &\;0\; & $2\LL+1$      & $3\LL^3+\LL^2+\LL+1$ & $4\LL^5+\LL^4+\LL^2+\LL+1$\\
\end{tabular}
\end{proof}
Now we return to the last summand in $(\ref{b16})$:
\begin{multline}
[Z']=[\cV(I_{n-1},I_{n},\tilde{G}_{n})]=[\cV(I_{n-1},I_{n},A_{n}^2I_{n-1}^1)]=
[\cV(I_{n-1},I_{n},A_{n})]+\\ [\cV(I_{n-1},I_{n},I_{n-1}^1)]-
[\cV(I_{n-1},I_{n},A_{n},I_{n-1}^1)]=[\cV(I_{n-1},I_{n},A_{n})]+\\
1+(\LL-1)[\cV(I_{n-1},I_{n},A_{n},I_{n-1}^1)]=2+\LL
y_{n-1}'+(\LL-1)[Z]. \label{b22}
\end{multline}
We use the same trick as appeared in \cite{BEK}. Define
$Z_1:=\cV(I_{n},I^1_{n-1})$ and $Z_2:=\cV(I_{n},I_{n-1})$. Then
$Z=Z_1\cap Z_2$. Recall that
\begin{equation}
    I^1_{n-1}I_{n-1}-S_{n-1}^2=I_{n}I^1_{n-2},
\end{equation}
where $S_{n-1}:=I_{n}(0;n-1)$, the determinant of the matrix when we
throw away 0-th column and (n-1)-th row. Define
\begin{equation}
    T:=Z_1\cup Z_2 = \cV(I_{n},S_{n-1})=\cV(I_{n},\prod_{i=1}^{n-1} A_i).
\end{equation}
Note that $Z_1\cong Z_2$ is a cone over $Y'_i$ in the setting of
Lemma \ref{recur1} and
\begin{equation}
    [Z]=2(1+\LL y'_{n-1})-[T].
    \label{b25}
\end{equation}
Define $T_I=T\cup\bigcap_{i\in I}\cV(A_i)$ for each
$I\subset\{1,\ldots,n-1\}$. Then
\begin{equation}\label{b125}
    [T]=\sum_i [T_i] - \sum_{i,j} [T_{i,j}]+ \sum_{i,j,k}
    [T_{i,j,k}]-\ldots + (-1)^{n-1} [T_{1,2,\ldots,n-1}].
\end{equation}
For $1\leq i_1<i_2<\ldots<i_p\leq n-2$,
\begin{equation}
    T_{i_1,\ldots,i_p}=\cV(I^{n-i_1}_{i_1}I^{n-i_2}_{i_2-i_1}\cdot\ldots\cdot
    I^{n-i_{p}}_{i_p-i_{p-1}}I_{n-i_{p}}).
    \label{b127}
\end{equation}
One can decompose $T_{i_1,\ldots,i_p}:=T'$ similar to (\ref{b125})
\begin{multline}
    T_{i_1,\ldots,i_p}=T'\cap\cV(I^{n-i_1}_{i_1})+
    T'\cap\cV(I^{n-i_2}_{i_2-i_1}) +\ldots+
    T'\cap\cV(I_{n-i_{p}})-\\ -
    T'\cap\cV(I^{n-i_1}_{i_1},I^{n-i_2}_{i_2-i_1})-\ldots+
    T'\cap\cV(I^{n-i_1}_{i_1},I^{n-i_2}_{i_2-i_1},I^{n-i_3}_{i_3-i_2})+\ldots
    \label{b128}
\end{multline}
Note that each of the first $p+1$ summands in the last sum is a cone
over some $I_j^k$, that is a determinant of a 3-diagonal matrix of
the right dimension similar to $Y_j$ in Lemma \ref{recur1}. Next,
$T'\cap\cV(I^{k}_{j},I^{k'}_{j'})$ is a cone over $Y_j\times
Y_{j'}$. And so on. Substituting all sums like (\ref{b127}) into
(\ref{b125}), we get an expression for $[T]$ as a big sum in terms
of $y_i$.


One likes to get a convenient recursive formula to compute the
number of points with a computer quickly. The good way to do this is
the following. Define $T^j=\cV(I_j,S_{j-1})$ for $2\leq j\leq n$, so
that $T^n=T$. Let $S_{(i)}$ be the sum consisting of the summands of
a big sum where $i_1=i$ in (\ref{b127}). First,
$[\cV(B_1I_{n-1})]\in S_{(1)}$. We write
\begin{multline}
    [\cV(B_1I_{n-1})]=(1+\LL[\cV(I_{n-1})])+
    ([\PP^{2n-4}]+\LL^{2n-3}[\cV(B_1)])-\\
    [\cV(B_1,I_{n-1})]=
    1+(\LL-1)y_{n-1}+([\PP^{2n-4}]+\LL^{2n-3}y_1).
\end{multline}
Next, consider a summand with exactly one $A_i$ vanishing (except
$A_1$). We get
\begin{multline}
    [\cV(B_1I^{n-i}_{i-1}I_{n-i})]=1+(\LL-1)[\cV(I^{n-i}_{i-1}I_{n-i})]+
    [\PP^{2n-5}]+\LL^{2n-4}y_1.
    \label{b130}
\end{multline}
In general, for
$P=I^{n-i_1}_{i_1-1}I^{n-i_2}_{i_2-i_1}\cdot\ldots\cdot
I^{n-i_{p}}_{i_p-i_{p-1}}I_{n-i_{p}}$ we compute
\begin{equation}
    [\cV(B_1P)]=1+(\LL-1)[\cV(P)]+[\PP^{2n-4-p}]+\LL^{2n-3-p}y_1.
\end{equation}
For a fixed $p$ we have exactly $\binom{n-2}{p}$ different summands
$P$. Taking a sum over all such $P$ for all $p$ with right signs, we
come to the sum
\begin{multline}
    S_{(1)}=\sum_{p=0}^{n-2}(-1)^p\binom{n-2}{p}+(\LL-1)y_{n-1}-(\LL-1)[T^{n-1}]+\\
    \sum_{p=0}^{n-2}(-1)^p\binom{n-2}{p}[\PP^{2n-4-p}]+
    y_1\sum_{p=0}^{n-2}(-1)^p\binom{n-2}{p}\LL^{2n-3-p}=\\
    (\LL-1)(y_{n-1}-[T^{n-1}])+\sum_{p=0}^{n-2}(-1)^p\binom{n-2}{p}[\PP^{2n-4-p}]+\\
    \LL^{n-1}(\LL-1)^{n-2}y_1.
    \label{b132}
\end{multline}
Consider now a summand $[\cV(I_{i}P)]$ of $S_{(i)}$ with
$P=I^{n-i_1}_{i_1-i}I^{n-i_2}_{i_2-i_1}\cdot\ldots\cdot
I_{n-i_{p}}$, $1<i<n-1$. Similar to (\ref{b130}), one obtains
\begin{multline}
    [\cV(I^{n-i}_iP)]=([\PP^{2i-2}]+\LL^{2i-1}[\cV(P)])+([\PP^{2n-2i-p-2}]+\\
    \LL^{2n-2i-p-1}[\cV(I^{n-i}_i)])- [\cV(I^{n-i}_i,P)]=([\PP^{2i-2}]-y_i)+\\
    (\LL^{2i-1}-y_i(\LL-1)-1)[\cV(P)]+[\PP^{2n-2i-p-2}]+\LL^{2n-2i-p-1}y_i.
    \label{b133}
\end{multline}
We used here that for two polynomials $P$ and $Q$ of different
non-intersected sets of variables (both of cardinality at least 2),
one has $[\cV(P,Q)]=(\LL-1)[\cV(P)][\cV(Q)]+[\cV(P)]+[\cV(Q)]$ in
the corresponding projective spaces. Taking the sum over all $p$ and
$P$ with the right signs and grouping like in (\ref{b132}), we
obtain
\begin{multline}
    S_{(i)}=(\LL^{2i-1}-y_i(\LL-1)-1)(y_{n-i}-[T^{n-i}])+\\
    \sum_{p=0}^{n-1-i}(-1)^p\binom{n-1-i}{p}[\PP^{2n-2i-2-p}]+
    \LL^{n-i}(\LL-1)^{n-1-i}y_i.
\end{multline}
for $i<n-1$ and
\begin{equation}
    S_{(n-1)}=\LL^{2n-3}y_1+1+(\LL-1)y_{n-1}+[\PP^{2n-4}].
\end{equation}

The big sum for $[T]=[T^n]$ is a sum of all $S_{(i)}$. Considering
$n$ to be non-fixed from the beginning, we get a recurrence formula
\begin{multline}
    [T^n]=\sum_{i=1}^{n-1} S_{(i)}=[\PP^{2n-4}]+\sum_{i=1}^{n-2}(\LL^{2i-1}-
    y_i(\LL-1)-1)\\(y_{n-i}-
    [T^{n-i}])+
    \sum_{i=1}^{n-2}\sum_{p=0}^{n-1-i}(-1)^p\binom{n-1-i}{p}[\PP^{2n-2i-2-p}]+\\
    \sum_{i=1}^{n-2}\LL^{n-i}(\LL-1)^{n-1-i}y_i+1+(\LL-1)y_{n-1}
\end{multline}
with $[T^2]=2$. We can simplify the sum of projective spaces
\begin{multline}
    \sum_{p=0}^{n-1-i}(-1)^p\binom{n-1-i}{p}[\PP^{2n-2i-2-p}]=
    [\PP^{2n-2i-2}]+\sum_{p=1}^{n-2-i}(-1)^p(\binom{n-2-i}{p}+\\
    \binom{n-2-i}{p-1})[\PP^{2n-2i-2-p}]+(-1)^{n-1-i}=
    \sum_{p=0}^{n-2-i}(-1)^p \binom{n-2-i}{p}([\PP^{2n-2i-2-p}-\\
    \PP^{2n-2i-3-p}]=\sum_{p=0}^{n-2-i}(-1)^p\binom{n-2-i}{p}\LL^{2n-2i-2-p}=
    \LL^{n-i}(\LL-1)^{n-i-2}
\end{multline}
for $1\leq i\leq n-2$. The formula for [T] simplifies to
\begin{multline}
    [T^n]=[\PP^{2n-4}]+\sum_{i=1}^{n-2}(\LL^{2i-1}-y_i(\LL-1)-1)
    (y_{n-i}-[T^{n-i}])+\\
    \sum_{i=1}^{n-2}\LL^{n-i}(\LL-1)^{n-2-i}(1+(\LL-1)y_i)+1+(\LL-1)y_{n-1}.
\end{multline}
First terms are
\begin{equation}
\begin{aligned}\label{b140}
[T^2]&=&2&,\\
[T^3]&=&4&\LL^2-\LL+2,\\
[T^4]&=&6&\LL^4 - 4\LL^3 + 5\LL^2 - 2\LL + 2,\\
[T^5]&=&8&\LL^6 - 8\LL^5 + 9\LL^4 - 7\LL^3 + 8\LL^2 - 3\LL + 2.
\end{aligned}
\end{equation}
By (\ref{b16}), (\ref{b22}) and (\ref{b25}), we obtain
\begin{multline}
    [X]=1+\LL+\LL^2[\PP^{2n-2}]-\LL^2-
    \LL^3y'_{n-1}+\\\LL^2(2+\LL y'_{n-1}+
    (\LL-1)(2(1+\LL y'_{n-1})-[T^{n}]))=\\
    1+\LL+\LL^2[\PP^{2n-2}]+\LL^2+2\LL^2(\LL-1)(1+\LL
    y'_{n-1})-\LL^2(\LL-1)[T^{n}].
\end{multline}
For example,
\begin{equation}
\begin{aligned}\label{b140}
[X_3]&=&\LL^4&+\LL^3+2\LL^2+\LL+1,\\
[X_4]&=&\LL^6&+\LL^5+4\LL^4-2\LL^3+2\LL^2+\LL+1, \\
[X_5]&=&\LL^8&+\LL^7+7\LL^6-8\LL^5+8\LL^4-3\LL^3+2\LL^2+\LL+1.
\end{aligned}
\end{equation}
for $X_k:=X_{WS_k}$.

\end{document}